\documentclass[preprint,12pt]{elsarticle}
\usepackage{amssymb}
\usepackage[fleqn]{amsmath}
\usepackage{amsthm}
\usepackage{tikz}
\usepackage{epstopdf}
\usepackage{lineno,hyperref}
\newtheorem{thm}{Theorem}
\newtheorem{cor}[thm]{Corollary}
\newtheorem{lem}[thm]{Lemma}
\newtheorem{rem}{Remark}
\theoremstyle{definition}
\newtheorem{defn}{Definition}
\theoremstyle{example}
\newtheorem{exm}{Example}
\newtheorem{prop}[thm]{Proposition}
\journal{https://arxiv.org}

\begin{document}

\begin{frontmatter}

\title{\Large A convexity of functions on convex
metric spaces of Takahashi and applications}

\author{Ahmed A. Abdelhakim}
\begin{abstract}
We quickly review and make some comments on the concept of convexity in metric spaces due to Takahashi. Then we introduce a concept of convex structure based convexity to functions on these spaces and refer to it as $W-$convexity. $W-$convex functions generalize convex functions on linear spaces. We discuss illustrative examples of
(strict) $ W-$convex functions and dedicate
the major part of this paper to proving a variety of properties that make them fit in very well with the classical theory of convex analysis. Finally, we apply some of our results to the metric projection problem and fixed point theory.
\end{abstract}


\end{frontmatter}
\section{Introduction and preliminaries}
There have been a few attempts
to introduce the structure
of convexity outside linear spaces.
Kirk \cite{kirk1,kirk2}, Penot \cite{Penot} and Takahashi \cite{Takahashi}, for example,
presented notions of convexity for sets in metric spaces. Even in the more general setting of topological spaces there is the work
of  Liepi\c{n}\v{s} \cite{Liepi} and
Taskovi\v{c} \cite{Taskovi}. Takahashi \cite {Takahashi}
introduced a general concept of
convexity that gave rise to what he referred to as convex metric spaces.
\begin{defn}\label{def1}(\cite{Takahashi}) Let $(X,d)\,$ be a metric space and
$I=[0,1].$ A continuous function
$\emph{W}:X\times X\times I\,\rightarrow\, X$ is said to be a
\emph{convex structure} on $X$ if for each $x,y\in X$ and all $
 t  \in I$,
\begin{eqnarray}\label{concon}
    d\left(u,\emph{W}(x,y; t  )\right)&\leq& (1- t )\, d(u,x)+ t\,  d(u,y)
\end{eqnarray}
for all $u\in X$. A metric space $(X,d)$ with a convex structure  ${W}$ is called a \emph{convex metric space} and is
denoted by $(X, W,d)$. A subset $C$ of $X$ is called
\emph{convex} if $\,W (x,y; t  )\in C\,$
whenever $\,x,y \in C\,$ and $\, t \in I$.
\end{defn}
What makes Takahashi's notion of convexity solid
is the invariance under taking intersections and
convexity of closed balls (\cite{Takahashi}, Propositions 1 and 2).
The convex structure $W$ in Definition \ref{def1} has the following property which is stated in \cite{Takahashi} without proof.
For the sake of completeness, we give a proof of it here.
\begin{lem}\label{113}
For any $x,y$ in a convex metric space $(X,W,d)$ and any $\, t \in I$ we have
\begin{eqnarray*}
      d\left(x,\emph{W}(x,y; t  )\right)\,=\, t \,d(x,y),\quad
  d\left(y,\emph{W}(x,y; t  )\right)\,=\,(1- t )\, d(x,y).
\end{eqnarray*}
\end{lem}
\begin{proof}
For simplicity, let $a$, $b$ and $c$
stand for $d\left(x,\emph{W}(x,y; t  )\right)$,
$d\left(y,\emph{W}(x,y; t  )\right)$ and $d(x,y)$ respectively.
Using (\ref{concon}) we get $\,\,a\leq \, t \, c\,\,$
and $\,\,b\leq \,(1- t )\,c$. But $\,\,c\,\leq \,a+b\,\,$ by
the triangle inequality. So
$\,\,c\,\leq\, a+b \,\leq \,(1- t ) \,c +
 t \,c=c.\,$
This means  $\,a+b=c\,$. If $\,a<t\,c\,$
then we would have $\,a+b<c\,$ which is a contradiction.
Therefore, we must have
$\,a= t \,c\,$ and consequently
 $\,b= (1- t )\, c.$
\end{proof}
The necessity for the condition (\ref{concon}) on $W$ to be a convex structure on a metric space $(X,d)$ is natural. To see this, assume that $ ( X, \parallel . \parallel_{X} ) $ is a normed linear space. Then the mapping $\,W:X\times X\times I\,\rightarrow\, X\;$ given by
\begin{align}\label{cs}
W(x,y; t  )=(1- t )\,x+ t\,y, \quad x,y \in X,\;\;  t  \in I,
\end{align}
defines a convex structure on $X$.
Indeed, if $\rho$ is the metric induced by the norm $\,\parallel . \parallel_{X}\,$ then
\begin{eqnarray*}
\rho \left(u,W(x,y; t )\right )&=&
\parallel u-\left((1- t )x + t  y \right) \parallel_{X} \\
&\leq&
(1- t ) \parallel u-x\parallel_{X} + \,
 t \parallel u-y\parallel_{X}\\
 &=&
(1- t )\,\rho(u,x) +  t \,\rho(u,y),\qquad \forall u\in X,\,  t \in I.
\end{eqnarray*}
\indent The picture gets clearer in the linear space ${\mathbb{R}}^2$ with the Euclidean metric
and the convex structure given by (\ref{cs}).
In this case, given two points $\,x,y\in {\mathbb{R}}^2\,$ and a $\, t  \in I$,
$\;z=W(x,y; t  )$ is a point that lies on
the line segment joining $x$  and $y$. Moreover, Lemma \ref{113} implies that if $\,\overline{xy}=L\,$
then $\,\overline{xz}= t  \,L\,$ and $\,\overline{zy}=(1- t ) \,L\,$ and we arrive at an
interesting exercise of elementary trigonometry
to show that $\,\overline{uz}\,\leq\,
(1- t )\,\overline{ux}+ t \,\overline{uy}\,$
for any point $u$ in the plane.
(Hint: Apply the Pythagorean theorem to the triangles
$\,\triangle uyv$, $\,\triangle uvz$ and
$\,\triangle uvx$ in the figure below then use the fact that
$\,\overline{xy}\,\leq\,\overline{xu}+\overline{uy}$).
 \begin{align*}
&\qquad\qquad  \begin{tikzpicture} [scale=8]
\draw (-0.25, 0.15) node[below left]
{$\dfrac{\overline{xz}}{\overline{zy}}\,=\,
\dfrac{ t }{1- t }$};
     \draw (0.0, 0) -- (0.7, 0) node[below] {$y$};
     \draw (0, 0) node[below left] {$x$};
     \draw (0.5, 0.3) node[above right] {$u$};
     \draw (0.5, 0) node[below] {$v$};
   \draw[dotted]
    (0.5, 0) -- (0.5, 0.3);
   \draw (0.25, 0.0) node[below] {$z=W(x,y; t  )$};
    \draw  (0.5, 0.3) -- (0.25, 0.0);
     \draw  (0.5, 0.3) -- (0.0, 0.0);
     \draw  (0.5, 0.3) -- (0.7, 0.0);
\end{tikzpicture}
\end{align*}
\indent Takahashi's concept of convexity was used extensively in fixed point theory in metric spaces (cf. \cite{Tbk} and the references therein). One of its most important applications is probably iterative approximation of fixed points in metric spaces.
There is quite huge literature on
fixed point iterations
(cf. \cite{berinde,chidume}).
Roughly speaking,
the formation of most, if not all, known fixed point iterative procedures is based on that of the Mann iteration \cite {mann} and the Ishikawa iteration \cite{ishikawa} as its very first generalization. All of these sequences require linearity and convexity of the ambient topological space. Although Takahashi's notion of convex metric spaces appeared in 1970, it was
not until 1988 that  Ding \cite{ding} exploited  it to construct a fixed point iterative sequence and proved a convergence theorem in a convex metric space.
To our best knowledge, this is the first time
a fixed point iteration, other than the well-known Picard iteration, was introduced
to metric spaces. Later, a lot of strong convergence results in convex metric spaces followed (see \cite{berinde}).\\
\indent In the light of Definition \ref{def1},
it is tempting to identify convex functions
on convex metric spaces. Based on the idea of convex structures on metric spaces, we define
and illustrate by examples what we call $W-$convex functions.
In linear metric spaces
with $W$ defined by (\ref{cs}),  $W-$convex functions coincide with traditional convex functions.
We show throughout the paper
that many of the main properties
of convex functions on linear spaces
are satisfied by $W-$convex functions.
As expected some of these properties
do not carry over automatically from linear spaces to convex metric spaces. In order to
achieve such properties we had to
require additional assumptions on the convex
structure $W.$ For instance, while
midpoint convex continuous functions
on normed linear spaces are convex,
midpoint $W-$convexity on its own
seems insufficient to obtain an analogous
result in convex metric spaces.
Another example appears when we
study the equivalence between
local boundedness from above and local Lipschitz continuity of $W-$convex functions.
To achieve this equivalence
we required the convex metric space to
satisfy a certain property that is naturally
satisfied in any linear space.
Other properties necessitated
providing a suitable framework to prove.
For example, to investigate the relation
between $W-$convexity of functions and the convexity of their epigraphs, we had to design
a convex structure on product metric spaces
to be able to define convex product metric spaces
and characterize their convex subsets.
Finally, we apply some of our results on
$W-$convexity to the metric projection problem and fixed point theory. For this purpose, we give a definition for strictly convex metric spaces
that generalizes strict convexity in Banach spaces
and relate it to a certain class of strictly $W-$convex functions.
\section{\large $W-$convex functions
 on convex metric spaces and their main properties}
\begin{defn}\label{defnconv}
A realvalued function $f$ on a convex metric space
$\left(X,\emph{W},d\right)$ is $W-$convex if
for all $x,y\in X$ and $ t  \in I,\,$
$f\left(\emph{W}\left(x,y; t \right)\right)
\,\leq\,(1- t )f(x)+ t  f(y).$ We call $f$
strictly $W-$convex if $\,f\left(\emph{W}\left(x,y; t \right)\right)
\,<\,(1- t )f(x)+ t  f(y)\,$
for all distinct points $x, y\in X$
and every $ t \in I^{\tiny o}=\,]0,1[.$
\end{defn}
\begin{exm} Consider the Euclidean space $\mathbb{R}^{3}$
with the Euclidean norm $\parallel.\parallel$. Let $\mathcal{B}$ be the subset of $\,\mathbb{R}^{3}\,$
that consists of all closed balls $B(\xi,r)$ with center $\xi \in \mathbb{R}^{3}$ and radius $r>0$. For any two balls
$\,\displaystyle B(\xi_{1},r_{1}),\,B(\xi_{2},r_{2})\in \mathcal{B},\,$
define the distance function
$\, d_{\mathcal{B}}\,\displaystyle  \left(B(\xi_{1},r_{1})
\textbf{,}\,B(\xi_{2},r_{2})\right)\,
\,=\,\parallel \xi_{1}-\xi_{2} \parallel+|r_{1}-r_{2}|$.
It is easy to check that $(\mathcal{B},d_{\mathcal{B}})$ is a metric space.
Let $\,\displaystyle W_{\mathcal{B}}:\mathcal{B}\times \mathcal{B}\times I\,\rightarrow\, \mathcal{B}\,$ be the continuous mapping given by
\begin{align*}
\hspace{-1 cm}W_{\mathcal{B}} \left(B(\xi_{1},r_{1}),\,B(\xi_{2},r_{2});
\theta\right)=
B\left((1-\theta)\xi_{1}+\theta \xi_{2},\,
(1-\theta)\, r_{1}+\theta \,r_{2}\,\right),
\quad \xi_{i} \in \mathbb{R}^{3},\, r_{i} >0,\, \theta \in I.
\end{align*}
Since for all $\,\theta \in I\,$ and any
three balls $\,B(\xi_{i},r_{i})\in \mathcal{B},\,$
 $i=1,2,3$,
\begin{eqnarray*}
&&\hspace*{-2 cm} d_{\mathcal{B}}\, \left(W_{\mathcal{B}} \left(B(\xi_{1},r_{1})\textbf{,}\,B(\xi_{2},r_{2});
\theta\right)\textbf{,}\,B(\xi_{3},r_{3})\right)\\
&=&
d_{\mathcal{B}}\, \left( B\left((1-\theta)\,\xi_{1}+\theta \,\xi_{2}\textbf{,}\,
(1-\theta)\, r_{1}+\theta \,r_{2}\right)\textbf{,}\,B(\xi_{3},r_{3})\right)\\
&=&\parallel (1-\theta)\,\xi_{1}+\theta \,\xi_{2}-\xi_{3} \parallel+|(1-\theta)\, r_{1}+\theta \,r_{2}-r_{3}|\\
&\leq&(1-\theta)\,\left(\,\parallel \xi_{1}-\xi_{3}\parallel +| r_{1}-r_{3}|\,\right)+
\theta\,\left(\,\parallel \xi_{2}-\xi_{3}\parallel +| r_{2}-r_{3}|\,\right)
\\
&=&(1-\theta)\,d_{\mathcal{B}} \left(B(\xi_{1},r_{1})
\textbf{,}\,B(\xi_{3},r_{3})\right)+
\theta\,d_{\mathcal{B}} \left(B(\xi_{2},r_{2})
\textbf{,}\,B(\xi_{3},r_{3})\right).
\end{eqnarray*}
Then $\left(\mathcal{B},W_{\mathcal{B}},d_{\mathcal{B}}\right)$ is a convex metric space. The function
$ f:\mathcal{B}\rightarrow \mathbb{R}$ defined by
$f\left(B(\xi,r)\right) := \,\parallel \xi \parallel+|r|$ is $W_{\mathcal{B}}-$convex.
\end{exm}
\begin{exm}
Let $\mathcal{I}$ be the family of closed intervals
$ [a,b] $ such that $0\leq a\leq b\leq1$
and define the mapping
 $W_{\mathcal{I}}:\mathcal{I}\times \mathcal{I} \times I$
by $W_{\mathcal{I}}(I_{i},I_{j}; t ):= \left[\left(1- t \right)a_{i}+ t  a_{j},
 \left(1- t \right)b_{i}+ t  b_{j}\right] $ for $I_{i}=[a_{i},b_{i}],\,I_{j}=[a_{j},b_{j}]\in
 \mathcal{I},$ $ t  \in I.$
If $d_{\mathcal{I}}$ is the Hausdorff
distance then $\left(\mathcal{I},
W_{\mathcal{I}},d_{\mathcal{I}}\right)$
is a convex metric space. This example
of a convex metric space is
given by Takahashi \cite{Takahashi}.\\
It is easy to verify that the Lebesgue measure defines
a $W_{\mathcal{I}}-$convex function
on  $\left(\mathcal{I},W_{\mathcal{I}},
d_{\mathcal{I}}\right)$.
\end{exm}
\begin{prop}\label{excon3}
\emph{(Composition with increasing convex functions).}
Assume that $f$ is a $W_{X}-$convex
 function on the convex metric space
 $(X,W_{X},d_{X}).$ Let
 $g:f(X)\rightarrow \mathbb{R}$
 be increasing and convex in the usual sense. Then
$ g\circ f $ is $W_{X}-$convex on $X$.
The composition $ g\circ f $ is strictly
$W_{X}-$convex if $g$ is strictly convex
or if $f$ is strictly
$W_{X}-$convex and $g$ is strictly increasing.
\end{prop}
\begin{proof}
Given $x,y \in X$ and $ t  \in I$,
in the light of Definition \ref{defnconv},
it follows from the monotonicity of $g$ that
\begin{align*}
g\left( f \left(W_{X}\left(x,y; t \right)\right)\right)
\,\leq\,g\left((1- t )f(x)+ t  f(y)\right)
\,\leq\,(1- t )\, g  \left(f(x)\right)+ t \, g \left(f(y)\right).
\end{align*}
\end{proof}
\begin{exm}\label{excon4}
Let $(X,W_{X},d_{X})$ be a convex metric space and let
$g:\mathbb{R} \rightarrow \mathbb{R} $
be increasing and (strictly) convex. Then
the function $f:X \rightarrow \mathbb{R}$
 defined by $ f(x):=g\left(d_{X}(x,x_{0}) \right) $ for some fixed $x_{0} \in X$
is (strictly) $W_{X}-$convex. Examples of
the function $g$ include $g(x)=x,$
$g(x)=\chi_{[0,\infty[}(x)\,x^{2},$
$g(x)=\chi_{[0,\infty[}(x)\,|x|$
 in the case of convexity and
$g(x)=e^{x},$
 $ g(x)=\chi_{[0,\infty[}(x)\,|x|^{\alpha} $
with $ \alpha>1 $ in the case of strict convexity.
\end{exm}
\begin{prop}\label{8tms}
\leavevmode
Let $\left(X,W,d\right)$  be a convex metric space.  Then
\begin{enumerate}
  \item The restriction $g$ of a $W-$convex function $f$ on $X$ to a convex subset $C$
of $X$ is also $W-$convex.
  \item\label{2} If $f$ is a $W-$convex function on $X$ and
$\alpha \geq 0$ then $\alpha f$ is also a $W-$convex function on $X.$
  \item\label{3}  The finite sum of $\,W-$convex functions on $X$ is $W-$convex.
\item\label{4} Conical combinations of $\,W-$convex functions is again $W-$convex.
  \item The maximum of a finite number of $\,W-$convex functions is $W-$convex.
  \item The pointwise limit of a sequence of $\,W-$convex functions is $W-$convex.
  \item Suppose that $\left(Y_{n}\right)$ is a sequence of convex subsets of $X$
and that $f_{n}$ is a $W-$convex function on $Y_{n},$ $n\geq 1.$ Let $S=\cap_{n} Y_{n}$ and
$M=\{x\in X: \sup_{n}f_{n}(x)< \infty\}.$
Then $M\cap S$ is convex and the
upper limit of the family $\left(f_{n}\right)_{n\geq 1},$
the function $f=\sup_{n} f_{n},$ is
 $W-$convex on it.
\item If $ f:X\rightarrow \mathbb{R}$ is
a nontrivial strictly
 $W-$convex function then $f$ has
at most one global minimizer on $X$.
\end{enumerate}
\end{prop}
\begin{proof}
\leavevmode
\begin{enumerate}
  \item
By the convexity of $C$, the restriction of $f$ to $C$ makes sense and the $W-$convexity of $g$ on $C$ follows from the $W-$convexity of $f$ on $X.$
  \item True since
$\;\;
\alpha f(W(x,y; t ))\leq
\alpha\left( (1- t )  f(x)+ t  f(y)\right)=
(1- t )  \alpha f(x)+ t  \alpha f(y).$
  \item  Obvious from Definition \ref{defnconv} and the linearity of the summation operator.
\item Follows from \ref{2} together with \ref{3}.
\item It suffices to show that $f=\max\{f_{1},f_{2}\}$ is $W-$convex on $X$ given the $W-$convexity of both $f_{1}$ and $f_{2}.$ For all $x,y \in X$ and $t \in I\,$ we have
\begin{align*}
\;f_{i}(\emph{W}(x,y;t))\leq (1-t)\,f_{i}(x)+t\,f_{i}(y)
\leq (1-t)\,f(x)+t\,f(y)
\end{align*}
 which yields  $\;f(\emph{W}(x,y;t))
\leq (1-t)\,f(x)+t\,f(y).$
\item A consequence of the monotonicity of the limit.
\item Let $x,y \in M \cap S.$
Then $x,y \in Y_{n}$ for all $n\geq1,$ $\sup_{n}f_{n}(x)< \infty$ and
 $\sup_{n}f_{n}(y)< \infty.$
Fix $t\in I$ and $n\geq1.$ By the convexity of $Y_{n}$ we know that it contains $\emph{W}(x,y;t).$
Hence $\emph{W}(x,y;t) \in S.$
To prove the convexity of $ M \cap S $
it remains to show that $\emph{W}(x,y;t) \in M.$
This follows from the $W-$convexity of $f_{n}$ as
$f_{n}\left(\emph{W}(x,y;t)\right) \leq
(1-t)\,\sup_{n} f_{n}(x)+t \,\sup_{n} f_{n}(y)<\infty.$
Finally, invoking the completeness axiom for the reals, the latter inequality implies
$\sup_{n} f_{n}\left(\emph{W}(x,y;t)\right) \leq
(1-t)\,\sup_{n}f_{n}(x)+t\, \sup_{n}f_{n}(y)<\infty,$
which shows that $\sup_{n} f_{n}$ is
$W-$convex on $M\cap S.$
\item Assume there are two
distinct points $x,y\in X$  such that
$ f(x)=f(y)=\inf_{x\in X}{f(x)}$. By convexity of
$X$ we have $W(x,y;\frac{1}{2})\in X$. Since $f$
is strictly $W-$convex then
$\,f\left(W(x,y;\frac{1}{2})\right)<
\frac{1}{2} f(x)+\frac{1}{2} f(y)=
\inf_{x\in X}{f(x)} $ which is a contradiction.
\end{enumerate}
\end{proof}
\section{$W-$convexity and continuity}
We begin with proving
 Lipschitz continuity
of $W-$convex functions on generalized
segments in convex metric spaces.
\begin{prop}\label{fpr}
Let $\left(X,W,d\right)$ be a convex metric space.
Let $x$ and $y$ be two distinct
points in $X$.
Then a $W-$convex function $f$
on the set $\,\mathcal{L}(x,y)=\{W(x,y;\lambda):
 0\leq \lambda \leq 1\}\,$
is Lipschitz continuous on it with a Lipschitz
constant that depends only on $x$ and $y$. Moreover,
if $\,|f(x)-f(y)|\,\leq\,\alpha\,d(x,y)\,$ for some
$\alpha>0$ then $\,|f(z)-f(w)|\,\leq\,\alpha\,d(z,w)\,$
for all $z,w\in \mathcal{L}(x,y)$.
\end{prop}
Before proceeding with
the proof of Proposition \ref{fpr},
we would like to make some remarks
on the set $\mathcal{L}(x,y)$.
\begin{rem}
If $X$ is a linear space and $W$ is defined by
(\ref{cs}) then $W(x,y;\lambda)$
is a unique vector in $X$ for each
 $\lambda\in I$ and the set $\mathcal{L}(x,y)$
 is known (\cite{cb1}) as the line segment joining
 the two vectors $x$ and
$y$. Clearly, the Euclidean geometry
justifies this notion. In metric spaces
the situation is different as,
for $\lambda \in I^{\tiny o},$
$W(x,y;\lambda)$ is not necessarily a unique
point. In fact the continuity of
$W$ in $\lambda$ required by Definition \ref{def1}  
is to be understood in the
sense of continuity of multivalued functions.
And if $\xi\in X$, the distance
$d\left(\xi, W(x,y;\lambda) \right)$
should be thought of as a point-set distance,
but this is just a technicality.
Nevertheless Lemma \ref{113} assures that
every point in the set $W(x,y;\lambda)$
belongs to $S\left(x,(1-\lambda)\,d(x,y)\right)
\cap S\left(y,\lambda\,d(x,y)\right)$
where $S(x_{0},r)$ is the usual sphere with center
$x_{0}$ and radius $r>0$.
Moreover, in the linear setting we have
the symmetry $W(x,y;\lambda)=W(y,x;1-\lambda)$
which leads to the symmetry
$\mathcal{L}(x,y)=\mathcal{L}(y,x)$.
While, from Definition \ref{def1} and
Lemma \ref{113} deduced from it, we have
\begin{align*}
\hspace{-0.5 cm}
d\big(W(x,y;\lambda),W(y,x;1-\lambda)\big)
\leq&\; (1-\lambda)\,
d\big(x,W(y,x;\lambda)\big)+
\lambda\,
d\big(y,W(y,x;\lambda)\big)\\
=&\;
\big((1-\lambda)^2+\lambda^2 \big)\,d(x,y).
\end{align*}
So, all that can be inferred
in the convex metric space
$\left(X,W,d\right)$
is
$\,d\big(W(x,y;\lambda),W(y,x;1-\lambda)\big)
< 2 d(x,y),$ $\lambda\in I^{\tiny o}.$
Consequently
$\mathcal{L}(x,y)$ is not to be assumed
symmetric in general. Finally, observe
that $\mathcal{L}(x,y)$ is closed.
Indeed, it follows from
Lemma \ref{113} that any $u\in \mathcal{L}(y,x)$
can be written as
$\,u=W(x,y,{d(x,u)}/{d(x,y)}).$
So if $\left(z_{n}\right)$ is a sequence of elements 
of $ \mathcal{L}(x,y)$ then
$\,z_{n}=W(x,y,{d(x,z_{n})}/{d(x,y)}),\,$
$n\geq 1$. If in addition
$z_{n}\rightarrow z$ as $n \rightarrow \infty$
then, by the continuity of $d$ and $W,$
we formally get $\,z=\lim_{n\rightarrow \infty}\,
 W(x,y,{d(x,z_{n})}/{d(x,y)})=
 W(x,y,{d(x,z)}/{d(x,y)}).\,$
Since $\,d(x,z_{n})\leq d(x,y)\,$ then,
passing to the limit, we also have
$\,d(x,z)\leq d(x,y).$
This shows that $z\in \mathcal{L}(x,y).$
\end{rem}
Now we prove Proposition \ref{fpr}.
\begin{proof}
Fix $x,y \in X$ so that
$d(x,y)>0$. Let $z,w\in \mathcal{L}(y,x) $
be such that $z\neq w$. Then, by $W-$convexity of $f$,
we have
\begin{eqnarray}
\label{z1} f\left(z\right)=
f\left( W\left(x,y,\frac{d(x,z)}{d(x,y)}\right)\right)
 &\leq& \left( 1-\frac{d(x,z)}{d(x,y)}\right) f(x)+
\frac{d(x,z)}{d(x,y)} f(y).
\end{eqnarray}
Similarly
\begin{eqnarray}
\label{w1} f\left(w\right)=
f\left( W\left(x,y,\frac{d(x,w)}{d(x,y)}\right)\right)
 &\leq& \left( 1-\frac{d(x,w)}{d(x,y)}\right) f(x)+
\frac{d(x,w)}{d(x,y)} f(y).
\end{eqnarray}
Considering (\ref{z1}) and (\ref{w1}),
we have only two possibilities. Either
\begin{align}
f\left(z\right)-f\left(w\right)
\nonumber
\leq&\; \left( d(x,y)  \right)^{-1}
\, \left(d(x,w)-d(x,z)\right)\,
 \left(  f(x)-f(y)\right)\\
\label{zw1}\leq&\;  \left( d(x,y)  \right)^{-1}
\,\left|  f(x)-f(y)\right|
\, d(z,w).
\end{align}
Or
\begin{align}
f\left(z\right)-f\left(w\right)
\nonumber
\leq&\; \left( d(x,y)  \right)^{-1}
\, \left(d(x,z)-d(x,w)\right)\,
 \left(  f(x)-f(y)\right)\\
\label{zw2}\leq&\;  \left( d(x,y)  \right)^{-1}
\,\left|  f(x)-f(y)\right|
\, d(z,w).
\end{align}
Interchanging $z$ and $w$ in
both sides of (\ref{zw1}) or (\ref{zw2})
we immediately get
\begin{align}\label{zwf}
|f\left(z\right)-f\left(w\right)|
\,\leq\, \left( d(x,y)  \right)^{-1}
\,\left|  f(x)-f(y)\right|
\, d(z,w)
\end{align}
which proves that $f$ is Lipschitz continuous
on $\,\mathcal{L}(y,x)\,$
with the Lipschitz constant \linebreak
$\left( d(x,y)  \right)^{-1}
\,\left|  f(x)-f(y)\right|$. The inequality
(\ref{zwf}) demonstrates
the second assertion of the proposition as well.
 \end{proof}
\begin{cor}
Let $\left(X,W,d\right)$ be a convex metric space.
If a $W-$convex function $f$
on the set $\,\mathcal{L}(x,y)=\{W(x,y;\lambda):
 0\leq \lambda \leq 1\}\,$
is such that $f(x)=f(y)$ then
$f$ is constant on $\mathcal{L}(x,y)$.
\end{cor}
Continuous functions on convex metric spaces are
$W-$convex provided that they are
midpoint $W-$convex in a certain sense.
 We prove this in the following Proposition.
\begin{prop}
Let $\left(X,W,d\right)$ be a convex metric space.
Every continuous function $f:X\rightarrow \mathbb{R}$ such that
$\,\displaystyle \,f\left( W \left(
x,y;\frac{\mu+\nu}{2}\right)\right)\leq
\frac{1}{2}\,f\left( W \left(
x,y;\mu\right)\right)+
\frac{1}{2}\,f\left( W \left(
x,y;\nu\right)\right),\,$
$x,y\in X,\,$ $\mu,\nu \in I,\,$
is $W-$convex.
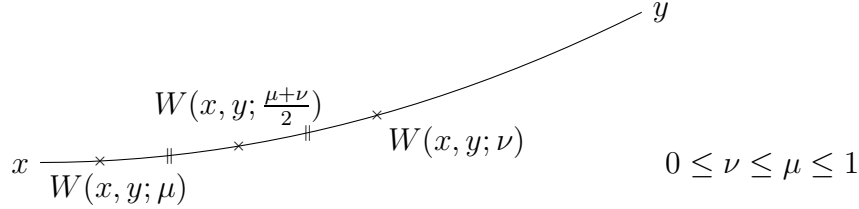
\begin{align*}
\begin{tikzpicture}[domain=0:1,xscale=8,yscale=2]
\draw plot (\x, {0.025+\x*\x}) node[right] {$y$};
\node [left] at (0,0) {$x$};
\node [below] at (0.13,0.03) {$W(x,y;\mu)$};
\node [above] at (0.33,0.2) {$W(x,y;\frac{\mu+\nu}{2})$};
\node [below right] at (0.56,0.3386) {$W(x,y;\nu)$};
\node at (0.1,0.035) {\textbf{\tiny $\times$}};
\node at (0.33,0.1339) {\textbf{\tiny$\times$}};
\node at (0.56,0.3386) {\textbf{\tiny$\times$}};
\node at (0.215,0.071225) {\textbf{\tiny$\parallel$}};
\node at (0.445,0.223025) {\textbf{\tiny$\parallel$}};
\node at (1.2,0.0) {$0\leq \nu \leq \mu  \leq 1$};
\end{tikzpicture}
\end{align*}
\end{prop}
\begin{proof}
Let $n$ be a nonnegative integer and let $\Lambda_{n}=\big\{{m}/{2^n}, m=0,1,...,2^{n}\big\}.$
By induction on $n$, we show that
\begin{align}\label{conlmd}
f\left( W(x,y;\lambda)\right)
\leq (1-\lambda)\,f(x)+\lambda\,f(y),
\quad \text{for every}\;x,y\in X,\;\lambda\in \Lambda_{n}.
\end{align}
Since, by lemma \ref{113},
$ x=W\left(x,y;0\right) $ and
$ y=W\left(x,y;1\right)$ then
(\ref{conlmd}) is valid when
$n=0$ as $\Lambda_{0}=\{0,1\}$.
Assume that (\ref{conlmd}) is
satisfied for any $\lambda \in \Lambda_{k}$
for some natural number $k$.
Now let $x,y\in X$ and suppose that $\lambda \in \Lambda_{k+1}$.
Obviously, there exist $s,t \in \Lambda_{k}$
such that $\lambda=(s+t)/2$. The induction hypothesis implies that
\begin{align}\label{indhypu1}
f\left( W(x,y;u)\right)
\leq (1-u)\,f(x)+u\,f(y),\quad u\in\{s,t\}.
\end{align}
By our assumption on $f$ we have
\begin{align}\label{indhypu2}
f\left( W(x,y;\lambda)\right)
\leq \frac{1}{2} f\left( W(x,y;s)\right)+
\frac{1}{2} f\left( W(x,y;t)\right).
\end{align}
Using (\ref{indhypu1}) in (\ref{indhypu2}) we obtain
\begin{align*}
f\left( W(x,y;\lambda)\right)
&\leq\, \frac{1}{2} \sum_{u\in\{s,t\}}\, \big( (1-u)\,f(x)+u\,f(y) \big) \\
&=\,\left(1-\frac{s+t}{2}\right)
\,f(x)+\frac{s+t}{2}\,f(y)
=\left(1-\lambda\right)
\,f(x)+\lambda\,f(y).
\end{align*}
This proves (\ref{conlmd}). Let $r\in I$ be arbitrary. Since the set $\Lambda=\cup_{n\geq 0} \Lambda_{n}$ is dense in $I$  then there exists a sequence $\left(r_{n}\right)\subset \Lambda$ that converges to $r$. Thus
\begin{align}\label{dnscnt}
f\left( W(x,y;r)\right) =
f\left( W(x,y;\lim_{n\rightarrow \infty} r_{n})\right)  = \lim_{n\rightarrow \infty}
f\left( W(x,y; r_{n})\right)
\end{align}
by the continuity of both the convex structure $W$ and the function $f$. Since $r_{n}\in \Lambda$ then
there exists an integer $m\geq 0$ such that
$r_{n}\in \Lambda_{m}$. By (\ref{conlmd}), $\,f\left( W(x,y; r_{n})\right)
\leq (1-r_{n})\,f(x)+r_{n}\,f(y)$. From the latter inequality, the monotonicity of the limit and (\ref{dnscnt}) we obtain
\begin{align*}
f\left( W(x,y;r)\right)\leq (1-\lim_{n\rightarrow \infty} r_{n})\,f(x)+ \lim_{n\rightarrow \infty} r_{n}\,f(y) = (1-r)\,f(x)+ r\,f(y).
\end{align*}
\end{proof}
The next lemma paves the way to Proposition
 \ref{boundcon} where we show that
boundedness of $\,W-$convex functions on certain convex metric spaces is a necessary and sufficient condition for their continuity. In fact, our discussion in the rest of this section is confined to convex metric spaces
$\left(X,W,d\right)$ that enjoy the property that
for every two distinct points $x,y\in X$
and every $\,\lambda \in\, ]0,1[\, $
there exists $\xi\in X$ such that
$x=W(y,\xi;\lambda)$ or there exists
$\eta\in X$  such that $ \,y=W(x,\eta;\lambda)$.
This property is naturally satisfied if
$X$ is a linear space with $W$ defined as in (\ref{cs}).
In that case,
$ \xi=\lambda^{-1}\left(x-y\right)+y$
and $\eta=\lambda^{-1}\left(y-x\right)+x$.
\begin{lem}\label{lemxst}
Let $B(x_{0},r)$ be an open ball centered at $x_{0}$ with radius $r>0$ that is contained in $X$.
If $f:X \rightarrow \mathbb{R}$ is $W-$ convex
such that $|f(x)|\leq M$ on $B(x_{0},r)$ then $f$ is $\displaystyle \frac{2M}{\rho}-$Lipschitz
on $B(x_{0},r-\rho),$ $\,0<\rho<r$.
\begin{proof}
Let $x$ and $y$  be two distinct points in
$B(x_{0},r).$ Then, by our assumption on
$\left(X,W,d\right)$, there exists $\xi\in X$ such that $\displaystyle x=W\left(y,\xi;\frac{d(x,y)}{\rho+d(x,y)}\right)$ or there exists $\eta\in X$  such that $ \,\displaystyle y=W\left( x,\eta;\frac{\rho}{\rho+d(x,y)}\right)$.
We shall deal with the first case and the second one can be treated analogously. First, since $f$
is $W-$convex then
\begin{align*}
f(x)\leq \frac{\rho}{\rho+d(x,y)} f(y)+
\frac{d(x,y)}{\rho+d(x,y)} f(\xi).
\end{align*}
This implies
\begin{align}\label{subsin}
f(x)-f(y)\leq \frac{f(\xi)-f(x)}{\rho} d(x,y).
\end{align}
Assume for the moment that $\xi\in B(x_{0},r)$.
Using the boundedness of $f$
on $B(x_{0},r)$ the inequality (\ref{subsin})
takes the form
\begin{align}\label{subsin1}
f(x)-f(y)\leq \frac{2M}{\rho} d(x,y).
\end{align}
Interchanging the roles of $x$ and $y$
then exploiting the symmetry of the metric, we deduce from (\ref{subsin1}) that
\begin{align*}
|f(x)-f(y)|\leq \frac{2M}{\rho} d(x,y).
\end{align*}
To complete the proof, it remains to
show that $\xi\in B(x_{0},r)$.
From Lemma \ref{113}, we have
\begin{align*}
d(\xi,x)=
d(\xi,W\left(y,\xi;\frac{d(x,y)}{\rho+d(x,y)}\right))
=\frac{\rho d(\xi,y)}{\rho+d(x,y)}
\leq \frac{\rho d(\xi,x)}{\rho+d(x,y)}
+\frac{\rho d(x,y)}{\rho+d(x,y)}.
\end{align*}
Solving this inequality for
$d(\xi,x)$ we find
$d(\xi,x)\leq \rho$. Finally
\begin{align*}
d(\xi,x_{0})\leq
d(\xi,x)+d(x,x_{0})<\rho+r-\rho=r.
\end{align*}
\end{proof}
\end{lem}
\begin{prop}\label{boundcon}
A $W-$convex function $f$ on $X$ is locally bounded if and only if it is locally Lipschitz.
\end{prop}
\begin{proof}
Of course a locally Lipschitz
function is continuous and therefore
locally bounded. Let $f$ be locally bounded and let
$x_{0}\in X$. Then there exists $r>0$ such that
$f$ is bounded on $B(x_{0},r)$ and,
by Lemma \ref{lemxst}, $f$ is Lipschitz
on $B(x_{0},r/2)$. Since $x_{0}$ was arbitrary
then $f$ is locally Lipschitz.
\end{proof}
\begin{rem}\label{bbove}
The local boundedness assumption on the $W-$convex function $f$ in Lemma \ref{lemxst},
and consequently in Proposition \ref{boundcon}, can be weakened to local boundedness from above.
To prove this, assume that
there exists $c>0$ such that $ f(\xi)\leq c $
 for every $\xi\in B(x_{0},r)$ and
let $x\in B(x_{0},r)$.
Then there exists $y\in X$ such that
$x_{0}=W(x,y,\frac{1}{2})$.
Since, by Lemma \ref{113}, $ d(y,x_{0})=d(x,x_{0})<r $ then $y\in B(x_{0},r)$.
Furthermore, by $W-$convexity of $f$,
$\,2f(x_{0})\leq f(x)+f(y)$. So
\begin{align*}
 2f(x_{0})-c \leq 2f(x_{0})-f(y) \leq f(x)\leq c\,\Longrightarrow\,
 |f(x)|\leq c+2|f(x_{0})|.
\end{align*}
\end{rem}
\begin{rem}
Recall that a function $f:X\rightarrow \mathbb{R}$
is lower semicontinuous at $x_{0}$ if for every
$t< f(x_{0})$ there is an open neighbourhood
$\mathcal{N}_{x_{0}}$ of $x_{0}$ such that
$f(x)>t$ for every $x\in \mathcal{N}_{x_{0}}$,
and if $\,\forall\:t> f(x_{0})\;\exists\;\mathcal{N}_{x_{0}}:\; f(x)<t,\;
\forall\:x\in \mathcal{N}_{x_{0}}$
then $f$ is upper semicontinuous at $x_{0}$.
It follows from Proposition \ref{boundcon}
is that upper semicontinuous $W-$convex
functions on open sets are continuous.
The same applies to lower semicontinuous
functions if and only if $X$ is complete.
Furthermore, a family of continuous
$W-$convex pointwise bounded functions
on an open convex subset of a complete metric space
is locally equi-bounded and locally equi-Lipschitz.
The most important consequence of these facts
is that pointwise convergence of sequences of
continuous $W-$convex functions on open convex
subsets of complete metric spaces is uniform
 on compact sets and preserves continuity.
Since the proofs of
these results  (cf. \cite{cb1,cb2,cb3} ) is indifferent to the topology
of the space and does not depend on linearity,
we find it redundant to give them here.
\end{rem}
\section{\large Epigraphs and sublevel sets of
$W-$convex functions}
The epigraph of a realvalued  function
$f$ on a set $C$ is the set $Epi (f)=\{(x,s)\in C\times \mathbb{R}: f(x)\leq s\}$  and the sublevel set of
$f$ of height $h$ is is the set
$ S_{h}(f)=\{x\in C: f(x)\leq h\}.$
In Proposition \ref{confunset} below
we show how $W-$convexity of functions
is related to the convexity of their
 epigraphs and sublevel sets.
First,
let $\left(X, W_{X}, d_{X}\right)$
and $\left(Y, W_{Y}, d_{Y}\right)$
be two convex metric spaces.
The mapping
$\,d_{p}: \left( X\times Y \right)^{2}
\rightarrow [0,\infty[$,
\begin{align*}
d_{p}\left((x_{1},y_{1}),(x_{2},y_{2})\right)=
\left\{
  \begin{array}{ll}
    \big( \left(d_{X}(x_{1},x_{2})\right)^{p}+
\left(d_{Y}(y_{1},y_{2})\right)^{p}\big)^{\frac{1}{p}}, & \hbox{$1\leq p<\infty$;} \\\\
\max\left\{d_{X}(x_{1},x_{2}),
d_{Y}(y_{1},y_{2})\right\}, & \hbox{ $p=\infty$,}
  \end{array}
\right.
\end{align*}
is a metric on the cartesian product
$X\times Y$ and $\left(X\times Y,d_{p} \right)$
is called a product metric space. Now,
let $\left(X,W_{X},d_{X}\right)$
and $\left(Y,W_{Y},d_{Y}\right)$ be two
convex metric spaces. We note the following:
\begin{lem}\label{lemconprod}
The mapping $\,W_{X\times Y}: \left(X\times Y\right)^{2}\times I\rightarrow X\times Y\,$ given by $\,W_{X\times Y}\left(
\left(x_{1},y_{1}\right),
\left(x_{2},y_{2}\right);t\right)=
\left( W_{X}(x_{1},x_{2};t),
W_{Y}(y_{1},y_{2};t)\right)$
is continuous and defines a convex structure on
the product metric space $\left(X\times Y,d_{1}\right).$
\end{lem}
\begin{proof}
The continuity of $W_{X\times Y}$ follows
 from the continuity of the convex structures $W_{X}$ and $W_{Y}.$ Let $ t\in I $ and $ (x_{i},y_{i})\in X \times Y,$ $i=1,2,3.$ By the definition of $W_{X\times Y},$ it remains to prove that
\begin{align}
\nonumber
& d_{1}\left( \left(x_{3},y_{3}\right),
\left(W_{X} \left(x_{1},x_{2};t\right),
W_{Y} \left(y_{1},y_{2};t \right)\right)
\right)\\ &\label{inefin} \qquad\qquad\qquad\leq\,
(1-t)\,d_{1}\left( \left(x_{1},y_{1}\right),
\left(x_{3},y_{3}\right)\right)+
t \,d_{1}\left( \left(x_{2},y_{2}\right),
\left(x_{3},y_{3}\right)\right).
\end{align}
However, we shall pretend
that we need to prove the inequality (\ref{inefin})
for the metric $d_{p}$ with $ 1\leq p<\infty.$
This enables us to demonstrate the difficulty in the proof for the case $p>1$ and explain
why the assertion of Lemma \ref{lemconprod}
is limited to the case $p=1.$ The metric $d_{\infty}$ is excluded for the same reason.
Of course we could simply construct counterexamples
for those cases but that would take us
outside the scope of this paper.
Now, for $1\leq p <\infty,$ we exploit the following facts:
\begin{description}
  \item (i)\hspace{0.2 cm}$W_{X}$ and $W_{Y}$ are convex structures on $X$
and $Y$ respectively.
  \item (ii)\hspace{0.2 cm}The map $x\mapsto x^{p}$ is monotonically increasing on $ [0,\infty[.$
  \item (iii)\hspace{0.2 cm}$(\mu+\nu)^p\leq 2^{p-1}\big(\mu^p+\nu^p\big),\,$
for all $\mu,\nu \geq 0.$
\end{description}
We then see that
\begin{align*}
\nonumber &\;d^{p}_{p}\left( \left(x_{3},y_{3}\right),
\left(W_{X} \left(x_{1},x_{2};t\right),
W_{Y} \left(y_{1},y_{2};t \right)\right)
\right)\\
\nonumber =&\;\big( d_{X}\left(x_{3},W_{X} \left(x_{1},x_{2};t\right)\right)\big)^{p}+
\big( d_{Y}\left(y_{3},W_{Y} \left(y_{1},y_{2};t\right)\right)\big)^{p}\\
\nonumber \leq &\;
\big( (1-t)d_{X}(x_{1},x_{3})+t d_{X}(x_{2},x_{3}) \big)^{p} +
\big( (1-t) d_{Y}(y_{1},y_{3})+t  d_{Y}(y_{2},y_{3}) \big)^{p}
\\
\nonumber \leq &\; 2^{p-1} (1-t)^{p}
\big[
\big(d_{X}(x_{1},x_{3})\big)^{p}+
\big(d_{Y}(y_{1},y_{3})\big)^{p}\big]+
2^{p-1} t^{p}
\big[ \big(d_{X}(x_{2},x_{3})\big)^{p}+
\big( d_{Y}(y_{2},y_{3})\big)^{p}
 \big]\\
= &\; 2^{p-1} \,\big[ (1-t)^{p}
d^{p}_{p}\big( \left(x_{1},y_{1}\right),
\left(x_{3},y_{3}\right)\big)+
 t^{p}
\,d^{p}_{p}\big( \left(x_{2},y_{2}\right),
\left(x_{3},y_{3}\right)\big) \big]
\end{align*}
which gives the desired inequality (\ref{inefin})
when $p=1.$
\end{proof}
Using Lemma \ref{lemconprod}, we can describe convex subsets
of convex product metric spaces.
\begin{defn}\label{zzcon}
A subset $Z$ of the convex product metric space
$\left(X\times Y, W_{X\times Y}, d_{1}\right)$
is convex if
$\,W_{X\times Y}\left(\left(x_{1},y_{1}\right),
\left(x_{2},y_{2}\right);t\right) \in Z$
for all points $(x_{1},y_{1}),(x_{2},y_{2})\in Z$
and all $t\in I.$
\end{defn}
In the light of Definition \ref{zzcon}
one can easily verify Lemma \ref{intersec} below.
\begin{lem}\label{intersec}
The intersection of any collection of
convex subsets of the convex product metric space
$\left(X\times Y, W_{X\times Y}, d_{1}\right)$
is convex.
\end{lem}
\begin{prop}\label{confunset}
 Let $f$ be a realvalued function
 on a convex metric space $\left(X,W_{X},d_{X}\right).$ Then
\leavevmode
\begin{enumerate}
  \item The function $f$ is $W_{X}-$convex if and only if $Epi (f)$ is a convex subset of the
convex product metric space $\left(X\times R, W_{X\times \mathbb{R}}, d_{X}+d_{\mathbb{R}}\right),$
where $W_{\mathbb{R}}$ and $d_{\mathbb{R}}$
are the usual convex structure and metric
on $\mathbb{R}$ respectively.
  \item If $f$ is $W_{X}-$convex then
 the sublevel set $S_{h}(f)$ is
a convex subset of $X$ for every $h\in \mathbb{R}.$
\end{enumerate}
\end{prop}
\begin{proof}
\begin{enumerate}
  \item Suppose that $f$ is $W_{X}-$convex
on $X$ and let $(x,s), (y,t)\in Epi(f).$
Then
\begin{align*}
f(W_{X}(x,y; \lambda ))\leq
(1- \lambda )  f(x)+ \lambda  f(y)
\leq (1- \lambda )  s+ \lambda  t
\end{align*}
for all $\lambda \in I.$
Therefore $\,\left(W_{X}(x,y; \lambda ),
(1- \lambda )  s+ \lambda  t\right) \in
Epi(f).$ That is
\begin{align*}
W_{X\times\mathbb{R}}
\left(\left(x,s\right),
\left(y,t\right);\lambda\right)=
\left(W_{X}(x,y;\lambda ),
W_{\mathbb{R}}(s,t;\lambda )\right)\in
Epi(f), \quad \lambda \in I.
\end{align*}
Hence $Epi (f)$ is a convex subset of $X\times R$. Conversely, suppose $Epi (f)$ is convex.
Fix $x,y\in X$ and $t\in I.$ Since $(x,f(x)),(y,f(y)) \in Epi (f),$ then
\begin{align*}
\left(W_{X} \left(x,y;t\right),
W_{\mathbb{R}}\left(f(x),f(y);t\right)\right)=
W_{X\times \mathbb{R}}\left(\left(x,f(x)\right),
\left(y,f(y)\right);t\right)\in Epi(f).
\end{align*}
Thus $\;f\left(W_{X} \left(x,y;t\right)\right)\leq
W_{\mathbb{R}}\left(f(x),f(y);t\right) =
(1-t)\,f(x)+t\,f(y),\;$ which is to say that
$f$ is $W_{X}-$convex.
\item Let $t \in I$ and let $x,y\in S_{h}(f)$ so that $f(x)\leq h$ and $f(y)\leq h.$
Since $f$ is $W_{X}-$convex then
$\,f\left(W_{X} \left(x,y;t\right)\right)\leq
(1-t)\,f(x)+t\,f(y) \leq h.$ Therefore
$W_{X} \left(x,y;t\right) \in S_{h}(f)$
and $S_{h}(f)$ is convex.
\end{enumerate}
\end{proof}
The following theorem is an application
of Lemma \ref{intersec}
and Proposition \ref{confunset}.
\begin{thm}
The pointwise supremum of an arbitrary collection of
$W-$convex functions is $W-$convex.
\end{thm}
\begin{proof}
Let $\left(X,W,d\right)$ be a convex metric space.
Let $J$ be some index set and assume that
$\{f_{i}\}_{i\in J}$ is a collection
of $W-$convex functions on $X$. Then, by
Proposition \ref{confunset},
$Epi (f_{i})$ is a convex subset of the
convex product metric space
 $\,\left(X\times R, W_{X\times \mathbb{R}},
 d_{X}+d_{\mathbb{R}}\right)\,$
for every $i\in J$.
If $\,f:X\rightarrow
\mathbb{R}\,$ is such that
$\,f(x)=\sup_{i\in J}{f_{i}(x)},\;x\in X,\,$
then $\,Epi(f)=\cap_{i \in J} Epi \left(f_{i}
\right).\,$ By Lemma \ref{intersec},
$\,Epi(f)\,$ is a convex subset of
 $\,\left(X\times R, W_{X\times \mathbb{R}},
 d_{X}+d_{\mathbb{R}}\right)\,$
and, using Proposition \ref{confunset}, it follows
that $f$ is $W-$convex on $X$.
\end{proof}
\section{\large Applications to the projection problem and fixed point theory}
Let $Y$ be a nonempty subset of a convex metric space $\left(X,W,d\right)$. The distance map
(cf. \cite{Papadopoulos}) $d_{Y}:X\rightarrow [0,\infty[$ is defined by $\,d_{Y}(x)=\inf_{y\in Y}{d(x,y)}.$ The distance map $d_{Y}$ is $W-$convex. Indeed, if $x_{1},x_{2}\in X,$ $y\in Y$ and $t\in I$ then, by the definition of $d_{Y}$, we have
\begin{align*}
d_{Y}(W\left( x_{1},x_{2}; t\right))\leq
\;d\left(W\left( x_{1},x_{2}; t\right),y\right)
\leq (1-t)\,d(x_{1},y)+t\,d(x_{2},y)
\end{align*}
for every $y\in Y$. Hence, by positive homogeneity and subadditivity of the infimum,
\begin{align*}
d_{Y}(W\left( x_{1},x_{2}; t\right))
\leq &\;\inf_{y\in Y}\big ((1-t)\,d(x_{1},y)+t\,d(x_{2},y)\big)\\
\leq &\;(1-t)\, \inf_{y\in Y}d(x_{1},y)+
t\, \inf_{y\in Y}d(x_{2},y)\\
= &\;(1-t)\, d_{Y}(x_{1})+t\, d_{Y}(x_{2}).
\end{align*}
If $Y$ is convex, then the metric projection operator (also called the nearest point mapping) (cf. \cite{Li})
$P_{Y}:X\rightarrow 2^{Y}$ is given by
$\;P_{Y}(x):=\big\{y\in Y: d(x,y)=d_{Y}(x) \big\}$.
If $P_{Y}(x)\neq \emptyset\,$ for every $x\in X$ then $Y$ is called proximal.
$P_{Y}(x)$ is convex (\cite{beg}, Lemma 3.2) and if $Y$ is closed then it is proximal. The proof of the proximality of $Y$ in this case is standard and given, in the setting of normed spaces, in many books  (cf. \cite{Papadopoulos,cb1}).
We briefly sketch it here. There exists a minimizing sequence $ (y_{n})\subset Y $ such that $d(x,y_{n})\rightarrow d_{Y}(x),\;x\in X, $ as $n\rightarrow \infty$. So the sequence $ (y_{n}) $ is bounded and, up to replacing it by a subsequence, it converges to $y$, say. Consequently, $d(x,y_{n})\rightarrow d(x,y)$
as $n\rightarrow \infty$. Hence $d(x,y)=d_{Y}(x)$.
Since $Y$ is closed then $y\in P_{Y}(x)$.\\
\indent The set of metric projections $P_{Y}(x)$, if nonempty, is not necessarily a singleton.
If $P_{Y}(x)$ is a singleton for each $x\in X$ then  the convex set $Y$ is called a Chebyshev set.
It is well-known (cf. \cite{cb3}) that
every closed convex subset of a strictly convex and reflexive Banach space is a Chebyshev set. \\
\indent We would like to describe sufficient conditions
for a point $x\in X$ to have a unique projection in $Y$.
We begin with introducing a definition for strict convexity in convex metric spaces.
\begin{defn}\label{dd11}
A convex metric
space $(X,W,d)$ is strictly convex
if for each $x_{0}\in X$ and any two distinct points $x,y\in S\left(x_{0},\rho\right)\;$
 with $\rho>0,\,$ we have $\,W(x,y;t)\in B\left(x_{0},\rho\right),\;\; \forall\;t \in I^{\tiny o}.$
\end{defn}
\begin{rem}
If $X$ is a linear space endowed with
a norm that induces the metric $d$ and
$W$ is given by (\ref{cs}) then Definition \ref{dd11},
after normalizing and translating to the origin,
coincides with the known definition of
strictly convex normed spaces \cite{cb3}.
\end{rem}
\begin{defn}\label{dd22}( (Strict) $W-$convexity w.r.t spheres).
Let $(X,W,d)$ be a convex metric space.
Fix $x_{0}\in X,$  $\rho>0$ and
 $ \sigma\in\,]0,\rho[.\,$
We call a realvalued function $f$ on
$ \overline{B\left(x_{0},\rho \right)}$
$W-$convex w.r.t the sphere $\,S\left(x_{0}, \sigma\right)\,$ if
\begin{align*}
f\left(W(x,y;t)\right)\leq
(1-t)\,f(x)+t\,f(y),\;\; \forall\; x,y\in
S\left(x_{0}, \sigma\right), \, t\in I,
\end{align*}
and we call it strictly $W-$convex w.r.t the sphere $\,S\left(x_{0}, \sigma\right)\,$ if
\begin{align*}
f\left(W(x,y;t)\right)<
(1-t)\,f(x)+t\,f(y),\;\; \forall\; x,y\in
S\left(x_{0}, \sigma\right)\;\text{with}\;
x\neq y,\;\forall\; t\in I^{\tiny o}.
\end{align*}
\end{defn}
The following proposition is a direct consequence of Definitions \ref{dd11} and \ref{dd22}.
\begin{prop}
Let $(X,W,d)$ be a convex metric space.
If for each $x_{0}\in X$ and $\rho>0,$
the function $f:X\rightarrow [0,\infty[\,$
defined by $\,f(x):=d(x,x_{0})\,$
is strictly $W-$convex w.r.t the sphere
$S\left(x_{0}, \rho\right)$ then the space
$X$ is strictly convex.
\end{prop}
The following theorem asserts that closed convex subsets of strictly convex metric spaces are
Chebyshev sets.
\begin{thm}\label{thm4}
Assume that $Y$ is a closed convex subset
of a strictly convex metric space $(X,W,d)$.
Then every $x\in X$ has a unique projection
on $Y$.
\end{thm}
\begin{proof}
Since $Y$ is closed then $\,P_{Y}(x)\neq \emptyset,\;\forall\:x \in X\,$ by the discussion above.
If $x\in Y$ then $P_{Y}=\{x\}.$ Let $x\in X-Y $
have two distinct projections $ y_{1},y_{2} \in Y$.
Then $\,d(x,y_{1})=d(x,y_{2})=d_{Y}(x)$. Let $t\in I^{\tiny 0}.$ Since $Y$ is convex then
$W(y_{1},y_{2};t) \in Y$, and since $X$ is
strictly convex then
\begin{align*}
d\left(W(y_{1},y_{2};t),x\right)
<(1-t)\,d(y_{1},x)+t\,d(y_{2},x)=
d_{Y}(x),
\end{align*}
which is a contradiction.
\end{proof}
\begin{thm}
Let $Y$ be a compact convex subset
of a strictly convex complete metric space. If $\,f: Y\rightarrow Y\,$ is continuous then it has a fixed point in $Y$.
\end{thm}
\begin{proof}
Since $Y$ is compact then it is closed
and, by Theorem \ref{thm4} above, $Y$ is a Chebyshev set.
The rest of the proof follows from Theorem 3.4 and Corollary 3.5 in \cite{beg}.
\end{proof}
\begin{thm}\label{thm555}
Let $(X,W,d)$ be a convex metric space
and let $\,T:X \rightarrow X\,$
is a nonexpansive mapping. Assume that the function $\,f:X\rightarrow [0,\infty[\,$ defined by $\,f(x):=d\left(x,Tx\right)$ is strictly $W-$convex with a local minimum at $\,\xi \in X$. Then $\xi$ is a fixed point of $ T$.
\end{thm}
\begin{proof}
By proposition \ref{8tms},
the point $\xi$ is the unique global minimizer
of $f$. Suppose that
$\,T \xi \neq \xi.$ Since $X$ is convex then
$W\left(\xi,T \xi; t\right) \in X\;\forall\;t\in I,\,$ and since $f$ is strictly $W-$convex on $X$ then, for all $\,t \in I^{\tiny o},\,$ we have
\begin{eqnarray}
\nonumber f\left(W\left(\xi,T \xi; t\right)\right)                &<& (1-t)\,f(\xi)+t\,f(T \xi) \\
\nonumber &=&  (1-t)\,d(\xi, T \xi)+t\,d(T \xi, T^{2} \xi)\\
\label{strff}
&\leq&  (1-t)\,d(\xi, T \xi)+t\,d( \xi, T \xi)
\,=\,d(\xi, T \xi)\,=\,f(\xi),
\end{eqnarray}
where we used nonexpansiveness of $f$
in estimating $\,d(T \xi, T^{2} \xi)\leq d(\xi, T \xi).$ The strict inequality (\ref{strff})
contradicts the fact that $\,f(\xi)=\min_{x\in X} f(x)$. Therefore we must have $\,T \xi= \xi.$
\end{proof}
\begin{rem}
The function $f$ is continuous
by the continuity of $\,T$. Hence, if
$X$ is compact then there does exist
a point $\xi \in X$ such that
$\,f(\xi)=\min_{x\in X} f(x),\,$
and we do not need to make such an
assumption on $f$.
\end{rem}


\end{document}